\documentclass[12pt]{article}

\usepackage{amssymb,amsthm,amsmath}

\usepackage{url}

\def\Z{\mathbb Z}

\newcommand{\R}{\mathbb{R}}

\newcommand{\K}{\mathbb{K}}
\newcommand{\m}{\mu}

\newtheorem{lemma}{Lemma}[section]

\newtheorem{theorem}[lemma]{Theorem}

\newtheorem{proposition}[lemma]{Proposition}

\newtheorem{question}[lemma]{Question}

\date{}

\title{Affine invariant points and new constructions}

\author{Ivan Iurchenko}

\newcommand\address{\noindent\leavevmode

\noindent
Ivan Iurchenko, \\
Dept.~of Math.~and Stat.~Sciences,\\
University of Alberta, \\
Edmonton, Alberta, Canada, T6G 2G1.\\
\texttt{\small e-mail:  iurchenk@ualberta.ca}

}

\begin{document}

\maketitle

\begin{abstract}
In~\cite{branko} Gr{\"u}nbaum asked if the set of all affine invariant points of a given convex
body is equal to the set of all points invariant under every affine automorphism of the body. In~\cite{ivan} we have proven the case of a body with no
 nontrivial affine automorphisms. After some partial results (\cite{aip},\cite{newaip}) the problem was
solved in positive by Mordhorst~\cite{Olaf}. In this note we provide an alternative proof of the affirmative answer, developing the ideas of~\cite{ivan}.
Moreover, our approach allows us to construct a new large class of affine invariant points.

\bigskip


\end{abstract}

\section{Introduction}
\label{intro}
Let $\K^n$ be the set of all convex bodies in $\R^n$ and let $P:\K^n\to\R^n$ be a function satisfying the following two conditions:

\noindent
1. For every nonsingular affine map $\varphi:\R^n\to\R^n$ and every convex body $K\in\K^n$ one has $P(\varphi(K))=\varphi(P(K)).$

\noindent
2. $P(K)$ is continuous in the Hausdorff metric.

\noindent
Such a function $P$ is called an {\it affine-invariant point}. The centroid and the center of the John ellipsoid (the ellipsoid of maximal volume contained
in a given convex body) are examples of affine-invariant points.

Let $\mathcal{P}$ be the set of all affine-invariant points in $\R^n.$ It was shown in~\cite{aip} that $\mathcal{P}$ is an affine subspace of the space of
continuous functions on $\K^n$ with values in $\R^n.$ Gr{\"u}nbaum~\cite{branko} asked
a natural question: how big is the set  $\mathcal{P}?$ In particular, how to describe the set $\mathcal{P}(K)=\{P(K)\mid P\in \mathcal{P}\}$ for a given $K\in\K^n?$  Denote the set of points fixed under affine maps of $K$ onto itself  by $\mathcal{F}(K).$
Gr{\"u}nbaum observed that $\mathcal{P}(K)\subset \mathcal{F}(K)$ and asked the following question:
\begin{question}\label{question}Is the set $\mathcal{P}$ big enough to ensure that $\mathcal{P}(K)= \mathcal{F}(K)$ for every $K\in\K^n?$
\end{question}
In~\cite{aip}, Meyer, Sch{\"u}tt and Werner proved that the set of convex bodies $K$ for which $\mathcal{P}(K)=\R^n$ is dense in $\K^n.$ Then the author showed that
if $\mathcal{F}(K)=\R^n$ then  $\mathcal{P}(K)=\R^n$~\cite{ivan}. Very recently, using a completely different approach, Mordhorst~\cite{Olaf} has shown the affirmative answer to the Question~\ref{question}. This proof used a previous development by P. Kuchment~\cite{k1,k2}. The purpose of this note is to show that the method of~\cite{ivan} can be also used to answer Question~\ref{question},
providing a new proof. Moreover, we construct a new large class of affine invariant points.

\section{Definitions and Notation}
\label{notat}

Recall some basic notations from group theory.

The group of all invertible linear transformations of $\R^n$ is denoted by $GL(n,\R).$ The group of all invertible linear transformations with the determinant equal to $1,$ i.e. the transformations which preserve volume and orientation is denoted by $SL(n,\R).$

For the purposes of the current paper we will use the group of all linear transformations preserving volume but not necessarily preserving orientation, i.e. the transformations with the determinant equal $\pm1$ denoted by $SL^-_n.$

The group of all affine transformations of $\R^n$ is denoted by $\mathit{Aff}(n).$ It may be represented as $GL(n)\ltimes \R^n$ with the rule $(r,x)(a)=r(a)+x$ where~$r\in GL(n),$ $x,a\in \R^n$.

The unit Euclidian ball in $\R^n$ is denoted by $B^n_2.$ The Euclidian norm of a vector is denoted by $\left|x\right|.$ The Lebesgue measure on $\R^n$ is denoted by $\m.$

A right (left) Haar measure is a measure on a locally compact topological group that is preserved under multiplication by the elements of the group from the right (left). The Lebesgue measure is an example of a Haar measure on $\R^n.$ Right and left Haar measures are unique up to multiplication however, not necessarily equal to each other. In this paper we always use a left Haar measure and denote the Haar measure of a set $X$ by $\mathrm{meas}(X).$

$\mathit{SAff}(n)$ is the group of all affine transformations of $\R^n$ preserving volume. This group may be represented as a  semidirect product of the group of all matrices with determinants equal to $\pm1$ and $\R^n$ with the rule $(r,x)(a)=r(a)+x$ for every $r\in GL(n) \ \text{with}\ \mathrm{det}(r)=\pm1, \ x\in \R^n$. $\mathit{SAff}(n)$ is equipped with the Haar measure, which is the product of Haar measures on the group of all matrices with the determinant equal to $\pm1$ and the group $\R^n.$

The Hausdorff metric is a metric on $\K^n,$ defined as
$$d_H(K_1,K_2)=\min\{\lambda\ge0:K_1\subset K_2+\lambda B_2^n ;K_2\subset K_1+\lambda B_2^n  \}.$$

By $\mathbb{K}^n_1$ we denote the set of all convex compact sets in $\R^n$ with volume $1$.

\section{Affine Invariant Points}

For a given convex body $K\in\K^n$ a family of affine invariant points is constructed by taking an arbitrary point $v$ and averaging all possible affine transformations of this point with the weight

$$F=F_K: \K^n\to C(SAff(n))$$
defined by
$$F_K(L)(\varphi)=\m(\varphi^{-1}(L)\cap K), L\in\K^n, \varphi\in SAff(x).$$

Let $k\ge1$ be an integer. For $L\in \K_1^n$ define the affine invariant point $T_{k,K,v}$ by
\begin{gather}\label{affp}
T_{k,K,v}(L)= \left({\int\limits_{SAff(n)}F^k(L)(\varphi)d\varphi}\right)^{-1}\int\limits_{SAff(n)}F^k(L)(\varphi)\varphi(v)d\varphi.
\end{gather}
In general, for $L\in\K^n$ we set

\begin{gather}
T_{k,K,v}(L)= |L|^{1/n}T_{k,K,v}(L/|L|^{1/n}).\label{affp2}
\end{gather}

\begin{theorem}\label{mainprop}
For a given convex body $K$  and a vector $v\in \mathcal{F}(K),$ the function $T_{k,K,v}:\K^n\to\R^n,$ defined in~\eqref{affp} has the following properties:

\noindent
1. There exists $k_0\in \Z_+$ such that for every $k\ge k_0,$ $T_{k,K,v}(L)$ is defined for all $L\in\K^n.$

\noindent
2. $T_{k,K,v}$ is an affine invariant point if defined.

\noindent
3. $T_{k,K,v}(K)\to v,\;k\to\infty.$
\end{theorem}

Theorem~\ref{mainprop} implies that for every $K\in \K^n$ and every $v\in\mathcal{F}(K)$ we can find an affine invariant point $F$ such that $F(K)$ is arbitrarily close to $v.$ However, this implies that every point in $\mathcal{F}(K)$ can be obtained as an affine point of $K$ because the set of all affine points is an affine space~\cite{aip}.

\section{Technical Part}

To prove Theorem~\ref{mainprop} we will require some tools for integration over the group $\mathit{SAff}(n).$

For a matrix $A\in GL(n,\R)$ the ordered sequence $\lambda_1\ge\lambda_2\ge\dots\ge\lambda_n>0$  of the singular values of the matrix A, is the sequence of all
eigenvalues of $\sqrt{AA^*}$ counting multiplicities; see e.g.~\cite{sing}.
In the case $A\in SL^{-}_{n}$ we have $1=\left|\mathrm{det}(A)\right| = \prod\limits_{i=1}^n \lambda_i.$
For a matrix  $A\in GL(n,\R)$ we denote by $\|A\|$ its operator norm $\ell_2\to\ell_2,$ that is
$$\|A\|=\sup\limits_{|x|=1}|Ax|.$$
Note that singular values of $A$ give a convenient description of the norm $\|A\| = \lambda_1.$

For $R\ge1$ the ``ball" $S_R$ is the set of all matrices $A\in SL^{-}_{n}$ such that $\|A\|\le R$.

Note that for $R_1, R_2\ge 1$ the following equality holds: $S_{R_1}S_{R_2} = S_{R_1R_2}.$ Indeed, by the property of the operator norm, $S_{R_1}S_{R_2} \subset S_{R_1R_2}.$ On the other hand, according to the polar decomposition, every $A\in S_{R_1R_2}$ may be represented in the form $A=UP,$ where $U$ is a unitary matrix and $P$ is positive Hermitian, see e.g.~\cite{polar}. Then $$A=UP^{\ln{R_1}/\ln{\left(R_1R_2\right)}}P^{\ln{R_2}/\ln{\left(R_1R_2\right)}},$$ with $UP^{\ln{R_1}/\ln{\left(R_1R_2\right)}}\in S_{R_1},\;P^{\ln{R_2}/\ln{\left(R_1R_2\right)}}\in S_{R_2}.$

\begin{lemma}\label{small}
For every $\varepsilon>0$ there exists a finite set $N\subset S_{2(1+\varepsilon)}$ such that for every integer $l\ge0$ one has
$$S_{2^l(1+\varepsilon)}\subset N^lS_{(1+\varepsilon)}.$$

\end{lemma}
\begin{proof}
Since the set $S_{2(1+\varepsilon)}$ is compact, it can be covered by some finite collection of balls:
 $$S_{2(1+\varepsilon)}\subset\cup_{N_i\in N}N_iS_{1+\varepsilon}=NS_{1+\varepsilon}.$$
We will show by induction that the set $N$ satisfies the condition of the proposition. The base case for $l=0$ is trivial. Now we show the inductive step:

\begin{align*}S_{2^{l+1}(1+\varepsilon)}=S_{2^l(1+\varepsilon)}S_{2}\subset N^lS_{1+\varepsilon}S_2= N^lS_{2(1+\varepsilon)}\subset N^lNS_{1+\varepsilon}.
\end{align*}
\end{proof}

\begin{proposition}\label{gewd}
For every $n\ge2,\;\alpha\ge0$ there exists $p\ge1$ such that for any convex bodies $K, L$ the integral

\begin{gather*}
\int\limits_{SL^{-}_{n}}\int\limits_{\vphantom{SL^{-}_{n}}\R^n}\m^p\left(L\cap (M(K)+x)\right)\|M\|^\alpha dxdM
\end{gather*}
converges. Here $dM$ is a Haar measure on $SL^{-}_{n}.$
\end{proposition}
\begin{proof}

There exists a radius $R>0$ such that the bodies $K, L$ are simultaneously contained within the ball $RB^n_2.$ Therefore,

\begin{align*}
&\int\limits_{SL^{-}_{n}}\int\limits_{\vphantom{SL^{-}_{n}}\R^n}\m^p\left(L\cap (M(K)+x)\right)\|M\|^\alpha dxdM
\\
&\le \int\limits_{SL^{-}_{n}}\int\limits_{\vphantom{SL^{-}_{n}}\R^n}\m^p\left(RB_2^n\cap (M(RB_2^n)+x)\right)\|M\|^\alpha dxdM
\\
&= R^{pn}\int\limits_{SL^{-}_{n}}\int\limits_{\vphantom{SL^{-}_{n}}\R^n}\m^p\left(B_2^n\cap (M(B_2^n)+\frac{x}{R})\right)\|M\|^\alpha dxdM
\\
&= R^{pn+n}\int\limits_{SL^{-}_{n}}\int\limits_{\vphantom{SL^{-}_{n}}\R^n}\m^p\left(B_2^n\cap (M(B_2^n)+x)\right)\|M\|^\alpha dxdM.
\end{align*}

It is enough to consider the convergence of the integral

\begin{gather}
\int\limits_{SL^{-}_{n}}\int\limits_{\vphantom{SL^{-}_{n}}\R^n}\m^p\left(B_2^n\cap (M(B_2^n)+x)\right)\|M\|^\alpha dxdM.\label{afp}
\end{gather}

Note that  the lengths of semiaxes of the ellipsoid $M(B_2^n)$ are defined by the singular values $\lambda_1\ge\lambda_2\ge\dots\ge\lambda_n$ of $M$ in particular, the diameter of $M\left(B_2^n\right)$ equals $2\lambda_1$ and its minimal width equals $2\lambda_n.$ This means that for $|x|>\lambda_1+1$ the volume $\m\left(B_2^n\cap (M(B_2^n)+x)\right)=0.$ For all other $x$ the ellipsoid $MB_2^n$ is contained within the slab $L=\{y\in\R^n:|\langle y,u\rangle|\le\lambda_n\}$ for some vector $u.$ Therefore,
$$\m\left(B_2^n\cap (M(B_2^n)+x)\right)\le \m\left(B_2^n\cap (L+x)\right)\le 2\lambda_n\left|B^{n-1}_2\right|.$$
Summing up, the integral~\eqref{afp} is bounded by

\begin{align*}\nonumber
&\int\limits_{SL^{-}_{n}}\int\limits_{\vphantom{SL^{-}_{n}}\R^n}\m^p\left(B_2^n\cap (M(B_2^n)+x)\right)\|M\|^\alpha dxdM\\
&=\int\limits_{SL^{-}_{n}}\int\limits_{\vphantom{SL^{-}_{n}}|x|\le\lambda_1+1}\m^p\left(B_2^n\cap (M(B_2^n)+x)\right)\|M\|^\alpha dxdM\\
&\le \int\limits_{SL^{-}_{n}}\int\limits_{\vphantom{SL^{-}_{n}}|x|\le\lambda_1+1}\left(2\lambda_n|B^{n-1}_2|\right)^p \|M\|^\alpha dxdM \\
&\le \int\limits_{SL^{-}_{n}}(2\lambda_1)^{n}|B_2^n|\left(2\lambda_n|B^{n-1}_2|\right)^p \|M\|^\alpha   dM \\
&=2^{n+p}\left|B_2^{n-1}\right|^p\left|B_2^n\right|\int\limits_{SL^{-}_{n}}\lambda_1^{n+\alpha}\lambda_n^p \label{afp2}dM.
\end{align*}
Keeping in mind that $\prod\limits_{i=1}^n \lambda_i = 1$ one has
$$\lambda_1^{n+\alpha}\lambda_n^p\le \lambda_1^{n+\alpha} \left(\lambda_2\lambda_3\dots\lambda_n\right)^{p/{(n-1)}}= \lambda_1^{n+\alpha}\left(\frac{1}{\lambda_1}\right)^{p/{(n-1)}}=\lambda_1^{n+\alpha-\frac{p}{n-1}}.$$

Finally, putting $q=-n-\alpha+\frac{p}{n-1}$ it is enough to show that there exists sufficiently big $q>0$ such that the integral
\begin{gather}
\int\limits_{SL^{-}_{n}}\|M\|^{-q} dM
\end{gather}
is convergent. To prove this we split the group $SL^{-}_{n}$ into smaller sets
$$S_{2^l}\setminus S_{2^{l-1}}, l\ge1.$$ Then

\begin{gather}\label{convergence}
\int\limits_{SL^{-}_{n}}\|M\|^{-q} dM = \sum\limits_{l=1}^{\infty}\int\limits_{S_{2^l}\setminus S_{2^{l-1}}}\|M\|^{-q}dM\le \sum\limits_{l=1}^\infty 2^{-lq}\mathrm{meas}(S_{2^{l}}).
\end{gather}
According to Lemma~\ref{small}, there exists a set $N$ such that \\ $\mathrm{meas}(S_{2^{l}})\le |N|^{l}\mathrm{meas}(S_{2(1+\varepsilon)}).$ Therefore, the series~\eqref{convergence} is bounded by a geometric series with the ratio $2^{-q}\left|N\right|$ which is convergent for $q>\log_2|N|.$
\end{proof}

\begin{proposition}\label{conv}
Let $G$ be a locally compact topological group and $dx$ be a Haar measure on $G.$ Let continuous functions $f, g$ satisfy the following conditions:

\noindent
\item[1.] For every $x\in G: 0\le f(x)\le 1.$

\noindent
\item[2.] There exists $x_0\in G$ such that $f(x_0)=1.$ Moreover, if $x_1\in G$ is such that $f(x_1)=f(x_0)=1$ then $g(x_1)=g(x_0).$

\noindent
\item[3.] There exists a constant $c<1$ and a compact $K$ such that for every $x\in G\setminus K,   f(x)<c.$

\noindent
\item[4.] There exists $k_0\ge 1$ such that for every $k\ge k_0$ the integrals
$$\int_G f^k(x)dx,\;\int_G f^k(x)\left|g(x)\right|dx$$ are convergent.

Then
 $$\lim_{k\to\infty}\frac{\int_G f^k(x)g(x)dx}{\int_G f^k(x)dx}=g(x_0).$$
\end{proposition}

  \begin{proof}

Note that the integral
$$\int_G f^k(x)|g(x)-g(x_0)|dx\le \int_G f^k(x)\left|g(x)\right|dx+ g(x_0)\int_G f^k(x)dx $$
is convergent for for $k\ge k_0.$
Passing to the new function $g-g(x_0)$ if needed, we may assume that $g(x_0)=0.$

The set  $N=f^{-1}(1)\subset K$ is closed and therefore compact. By the assumption of the proposition $g(N)=\{0\}.$ Fix $\varepsilon>0$ and consider a neighbourhood $U$ of  $N$ such that $|g|<\varepsilon$ on $U.$ There exists a positive constant $C<1$ such that $f<C$ outside of $U.$ Indeed, outside of $K$ the function $f$ is bounded from above by $c,$ on the compact set $K\setminus U$ the function $f$ is separated from $1$ by the compactness argument. By continuity of $f,$ there exists a constant $D\in(C,1)$ and a neighborhood $V\subset U$ of $N$ such that $D<f\le1$ on $V.$ Then

\begin{align*} \phantom{\le}\frac{\int_G f^k(x)\left|g(x)\right|\,dx}{\int_G f^k(x)dx}
 \le& \frac{\int_U f^k(x)\left|g(x)\right|dx+\int_{G\setminus U} f^k(x)\left|g(x)\right|dx}{\int_U f^k(x)\,dx}\\
 \le& \varepsilon+\frac{\int_{G\setminus U} f^k(x)\left|g(x)\right|\,dx}{\int_V f^k(x)\,dx} \\
 \le& \varepsilon+\frac{C^{k-k_0}\int_{G\setminus U} f^{k_0}(x)\left|g(x)\right|\,dx}{D^{k-k_0}\int_V f^{k_0}(x)\,dx}\to \varepsilon,\;k\to\infty.
\end{align*}
Sending $\varepsilon$ to $0$ we obtain the required statement.

\end{proof}

\begin{proof}[Proof of Theorem~\ref{mainprop}.]

For fixed $K$ and $v$ we will shorten the notation by writing $T_k$ instead of $T_{k,K,v}.$

\noindent
1. Proposition~\ref{gewd} applied with $\alpha=1$ (respectively, $\alpha=0$) implies that the integral in the numerator (respectively, denominator) is convergent.

\noindent
2. $T_k(cK)=cT_k(K)$ by the definition of $T_k.$\

\noindent
For every $\tau\in SAff(n)$ and $L\in\K^n_1:\; T_k(\tau(L))=\tau(T_k(L)).$

\noindent
Denote $$c=\left({\int_{SAff(n)}F^k(L)(\varphi)d\varphi}\right)^{-1}.$$ For arbitrary $\tau\in SAff(n)$ we have

\begin{align*}
T(\tau L)=&c \int\limits_{SAff(n)} F^k(\tau L)(\varphi)\varphi(v)d\varphi=c\int\limits_{SAff(n)}F^k(L)(\tau^{-1}\varphi)\varphi(v)d\varphi.
\end{align*}

\noindent
Replacing $\varphi$ by $\tau\varphi$ we get

\begin{align*}
T(\tau L)=c\int\limits_{SAff(n)} F^k(L)(\varphi)\tau(\varphi(v))d\varphi =\tau\left(c\int\limits_{SAff(n)} F^k(L)(\varphi)\varphi(v)d\varphi\right).
\end{align*}

The last equality holds because $cF^k(L)(\varphi)d\varphi$ is a probabilistic measure. Therefore, for every affine $\tau$ and every integrable function $f$ one has
$$\int\limits_{SAff(n)} \tau(f(\varphi))cF^k(L)(\varphi)d\varphi= \tau\left(\int\limits_{SAff(n)} f(\varphi)cF^k(L)(\varphi)d\varphi\right).$$

\noindent
Note that the function $$\frac{1}{\int_{S_R\times\R^n}F^k(L)(\varphi)d\varphi}\int\limits_{S_R\times\R^n}F^k(L)(\varphi)\varphi(v)d\varphi$$
is continuous as a function of $L$ by the Lebesgue's dominated convergence theorem because both integrals are uniformly bounded by a convergent integral by Proposition~\ref{gewd}. Then
$$T_k = \lim\limits_{R\to\infty}\frac{1}{\int_{S_R\times\R^n}F^k(L)(\varphi)d\varphi}\int\limits_{S_R\times\R^n}F^k(L)(\varphi)\varphi(v)d\varphi$$
is continuous.

\noindent
3. Convergence is the direct application of the Proposition~\ref{conv}, where  $f(\varphi)=F(K)(\varphi)$ and $g(\varphi)$ is $\varphi(v)$ taken coordinatewise. Similarly to the proof of the Proposition~\ref{gewd} the function $F(K)((A,x))$ is separated from $1$ when either $\|A\|$ or $\left|x\right|$ is big. Note that $F(K)(id)=1$ and if  $F(K)(\varphi)=1$ then $\varphi(K)=K$ which means $\varphi(v)=v$ because $v\in\mathcal{F}(K).$ Therefore,

$$\frac{1}{\int_{SAff(n)}F(K)^k(\varphi)d\varphi}\int\limits_{SAff(n)}F(K)^k(\varphi)\varphi(v)d\varphi\to id(v) = v,\;k\to\infty.$$
\end{proof}

%

\vspace{1cm}

\address

\end{document}